\documentclass[11pt]{article}

\usepackage{amsmath}
\usepackage{amsthm}
\usepackage{amssymb}
\usepackage{amscd}
\usepackage{pdfsync}
\usepackage{epsfig}
\usepackage{graphicx,color} 
 
 \newtheorem{theorem}{\sc Theorem}[section]
 
 \newtheorem{cor}[theorem]{\sc Corollary}
 \newtheorem{prop}[theorem]{\sc Proposition}

\newtheorem{remark}[theorem]{\sc Remark}

\usepackage{latexsym}
\usepackage{url}

\def\PG{\mathrm{PG}}

\def\F{\mathbb{F}}
\def\Fq{\mathbb{F}_q}

\def\cD{\mathcal{D}}

\def\GL{\mathrm{GL}}

\def\dim{\mathrm{dim}}

\def\cF{\mathcal{F}}
\def\cB{\mathcal{B}}
\def\cA{\mathcal{A}}
\def\cH{\mathcal{H}}

\DeclareMathOperator{\subl}{subl}

\DeclareMathOperator{\wt}{wt}

\DeclareMathOperator{\PGL}{PGL}
\DeclareMathOperator{\PGGL}{P\Gamma L}
\DeclareMathOperator{\AGL}{AGL}
\DeclareMathOperator{\AG}{AG}

\def\cP{\mathcal{P}}

\usepackage{tikz}
\usetikzlibrary{matrix,decorations.pathreplacing}

\begin{document}

\title{Subgeometries and linear sets on a projective line}
\author{Michel Lavrauw\thanks{Dipartimento di Tecnica e Gestione dei Sistemi Industriali,
Universit\`a di Padova, Stradella S. Nicola 3, 36100 Vicenza, Italy, e-mail: michel.lavrauw@unipd.it. 
The research of this author is supported by the Research Foundation Flanders-Belgium (FWO-Vlaanderen) and by a Progetto di Ateneo from Universit\`a di Padova (CPDA113797/11).} 
\ and Corrado Zanella\thanks{Dipartimento di Tecnica e Gestione dei Sistemi Industriali,
Universit\`a di Padova, Stradella S. Nicola 3, 36100 Vicenza, Italy, e-mail: corrado.zanella@unipd.it.
The research of this author is supported by the Italian Ministry of Education, University and Research (PRIN 2012 project ``Strutture geometriche, combinatoria e loro applicazioni'').}}
\date{\today}
\maketitle
\begin{abstract}
We define the splash of a subgeometry on a projective line, extending the definition of \cite{BaJa13} to general 
dimension and prove that a splash is always a linear set. We also prove the converse: each linear set on a projective line is the splash of some subgeometry. Therefore an alternative description of linear sets on a projective line is obtained.
We introduce the notion of a club of rank $r$, generalizing the definition from \cite{FaSz2006}, and show that 
clubs  correspond to tangent splashes. We obtain a condition 
for a splash to be a scattered linear set and
give a characterization of clubs, or equivalently of tangent splashes.
We also investigate the equivalence problem for tangent splashes and determine a necessary and sufficient condition for two 
tangent splashes to be (projectively) equivalent.

A.M.S. CLASSIFICATION: 51E20

\end{abstract}

\section{Introduction and motivation}\label{sec:intro}

Given a subgeometry $\pi_0$ and a line $l_\infty$ in a projective space $\pi$, by extending the hyperplanes of $\pi_0$ to hyperplanes of $\pi$ and intersecting these with the line $l_\infty$, one obtains a set of points on the projective line $l_\infty$. 
Precisely, if we denote the set of hyperplanes of a projective space $\pi$ by ${\mathcal{H}}(\pi)$, and $\overline{U}$ denotes the extension of a subspace $U$ of the subgeometry $\pi_0$ to a subspace of $\pi$,  then we obtain the set of points $\{ l_\infty \cap \overline{H} ~:~H \in {\mathcal{H}}(\pi_0)\}$.
These sets have been studied in \cite{BaJa13} and \cite{BaJa14} for Desarguesian planes and cubic extensions, i.e. for a subplane $\pi_0\cong\PG(2,q)$ in $\pi\cong\PG(2,q^3)$, where such a set is called 
the {\it splash of $\pi_0$ on $l_\infty$}. If $l_\infty$ is tangent (respectively external) to $\pi_0$, then a splash is called the {\it tangent splash} (respectively {\it external splash}) 
{of $\pi_0$ on $l_\infty$}. Note that when $l_\infty$ is secant to $\pi_0$, the splash of $\pi_0$ on $l_\infty$ is just a subline.
We study the splash of a subgeometry $\PG(r-1,q)$ in $\PG(r-1,q^n)$ on a line $l_\infty$.

The article is structured as follows. In Section \ref{sec:prelim} we collect the necessary definitions and notation in order to make the paper self contained and accessible. In Section \ref{sec:equiv} we show the equivalence between splashes and linear sets on a projective line (Theorem \ref{thm:equivalence}) and prove that the weight of a point of the linear set is determined by the number of hyperplanes through that point, leading to a characterisation of scattered linear sets. In Section \ref{sect:Characterization} we obtain a geometric characterisation of so-called clubs or equivalently of tangent splashes, and count the number of distinct 
tangent splashes
in $\PG(1,q^n)$. 
We conclude with Section \ref{sec:projective_equivalence}, where we study the projective equivalence of tangent splashes.

This work is motivated by the link between splashes and linear sets on a projective lines.  The concept of a splash of a subplane, although quite a natural geometric object to consider, has
been studied only recently, see \cite{BaJa13,BaJa14}.
This paper extends the definition of a splash from subplanes to 
subgeometries of order $q$ in higher dimensional projective spaces, and from cubic to general extension fields. 
Moreover, this generalization leads to a new interpretation of linear sets on a projective line.
The equivalence stated in Theorem 3.1 may turn out useful in investigating linear sets,
for instance by linking them to certain ruled surfaces in affine $(2n)$-dimensional spaces over $\Fq$,  relying on results from \cite{BaJa13,BaJa14}.
 Linear sets and field reduction have played an important role 
in the construction and characterization of many objects in finite geometry in recent years. The reader is referred to \cite{Po10} and \cite{LaVaPrep} for surveys and further references.

\section{Preliminaries}\label{sec:prelim}

In this section we collect the definitions and notation that will be used throughout the article. The finite field of order $q$ will be denoted by $\F_q$. The projective space associated with a vector space $U$ will be denoted by $\PG(U)$. The $(r-1)$-dimensional projective space over the field $\F$ will be denoted by $\PG(\F^r)$ or $\PG(r-1,q)$ in case $\F=\F_q$. 
The sets of points, lines and hyperplanes of a projective space $\pi$ will be denoted by ${\mathcal{P}}(\pi)$, ${\mathcal{L}}(\pi)$ and ${\mathcal{H}}(\pi)$, respectively;
but we will often write $\pi$ instead of $\cP(\pi)$ when the meaning is clear.
A {\it subgeometry} of a projective space $\PG(\F^r)$ is the set $S$ of points for which there exists a frame with respect to which the homogeneous coordinates of points in $S$ take values from a subfield $\F_0$ of $\F$, together with the subspaces generated by these points over $\F_0$. 
A subgeometry $\pi_0$ of $\PG(\F^r)$ is then isomorphic to $\PG(\F_0^r)$. 
If $\F_0$ has order $q$, then $\pi_0$ is called a {\it subgeometry of order $q$},
or a \textit{$q$-subgeometry}. 
A $k$-dimensional projective subspace $U$ of a subgeometry 
$\pi_0\cong \PG(r-1,q)$ of $\pi\cong \PG(r-1,q^n)$ generates 
a $k$-dimensional subspace of $\PG(r-1,q^n)$ ($-1\le k<r$)
(called the \textit{extension} of $U$). 
If there is no ambiguity we will denote both subspaces by $U$, otherwise we might use $U$ for the $\F_q$-subspace and $\overline{U}$ for the $\F_{q^n}$-subspace.
For $k=1$ or $k=2$ a $k$-dimensional subspace of $\pi_0$ as above is also called a
$q$-\textit{subline} or a $q$-\textit{subplane} of $\pi$.

Let $\pi_0$ be a $q$-subgeometry in $\pi\cong \PG(r-1,q^n)$, $r\geq 2$, $n>1$,
and consider a line $l_\infty$ of $\pi$ not contained in the extension of a hyperplane of $\pi_0$. 
We define the {\it splash of $\pi_0$ on $l_\infty$} as the set of points of $l_\infty$ which are contained in a subspace spanned by points of $\pi_0$. We denote this set by $S(\pi_0,l_\infty)$. If the line $l_\infty$ is external to $\pi_0$, then $S(\pi_0,l_\infty)$ is called an {\it external splash}, and if $l_\infty$ is tangent to $\pi_0$, then $S(\pi_0,l_\infty)$ is called a {\it tangent splash}.
The \textit{centre} of a tangent splash $S(\pi_0,l_\infty)$ is the intersection point $\pi_0\cap l_\infty$.

We will use the same notation $\cF_{r,n,q}$ as in \cite{LaVaPrep} for the {\it field reduction map} from $\pi=\PG(r-1,q^n)$ to $\PG(rn-1,q)$. The image of ${\mathcal{P}}(\pi)$ under $\cF_{r,n,q}$ is a {\it Desarguesian spread} $\cD_{r,n,q}$ of $\PG(rn-1,q)$. 
We note that the elements of $\cD_{r,n,q}$ are $(n-1)$-dimensional subspaces of $\pi$ and they form a partition of $\PG(rn-1,q)$. 
The field reduction map induces a bijection between the set of points of $\PG(r-1,q^n)$ and the set of elements of $\cD_{r,n,q}$. 
If $T$ is a subset of $\PG(rn-1,q)$, then the set of elements of $\cD_{r,n,q}$ which have non-empty intersection with $T$ will be denoted by $\cB(T)$, i.e. 
\begin{eqnarray}
\cB(T):=\{R\in \cD_{r,n,q}~:~R\cap T \neq \emptyset\}.
\end{eqnarray}
The inverse image of the set $\cB(T)$ under the field reduction map $\cF_{r,n,q}$, is a set of points of $\PG(r-1,q^n)$, which by abuse of notation will also be denoted by $\cB(T)$.
Moreover if $V$ is a subspace of the underlying vector space then 
$\cB(\PG(V))$ will also be denoted by $\cB(V)$.
An {\it $\F_q$-linear set} of $\PG(r-1,q^n)$ is a set of points $L$ for which there exists a subspace $U$ of $\PG(rn-1,q)$ such that $L=\cB(U)$. Given a linear set $L=\cB(U)$, we say that the {\it rank} of $L$ is $\dim(U) +1$ and the {\it weight} of a point $x\in L$ is defined as
$\dim(\cF_{r,n,q}(x)\cap U)+1$.
For more on field reduction we refer to \cite{LaVaPrep}.

A linear set $L$ is called {\it scattered} if every point of $L$ has weight one. Scattered linear sets are equivalent to scattered subspaces with respect to a Desarguesian spread and they were introduced in \cite{BlLa2000}. Scattered linear sets were further studied in \cite{LuMaPoTrPrep} and \cite{LaVa2013}.
We call a linear set $L$ a {\it club} if $L$ has rank $r\geq 3$, and a point of $L$ has weight $r-1$;
consequently, all other points of $L$ have weight one. 
This generalizes the definition of a club as introduced in \cite{FaSz2006} from $r=3$ to $r\geq 3$. Clubs and scattered linear sets on the projective line have been studied in \cite{LaVV10}.

%

\section{Equivalence of linear sets and splashes}\label{sec:equiv}

\begin{theorem}\label{thm:equivalence}
Let $r,n>1$.
If $S=S(\pi_0,l_\infty)$ is the splash of the $q$-subgeometry $\pi_0$ of $\PG(r-1,q^n)$ on the line $l_\infty$, 
then $S$ is an $\Fq$-linear set 
of rank $r$. Conversely, if $S$ is an ${\mathbb{F}}_q$-linear set of rank $r$ on the line $l_\infty\cong\PG(1,q^n)$, then there exists an embedding of
$l_\infty$ in $\PG(r-1,q^n)$ and a $q$-subgeometry $\pi_0$ of $\PG(r-1,q^n)$ such that $S=S(\pi_0,l_\infty)$.
\end{theorem}
\begin{proof}
The proof is based on the following three observations:\\
(i) in a finite projective space of dimension at least two, a set of hyperplanes is called \textit{linear}
if it is linear in the dual space;\\
(ii) given a line $l_\infty$ and an $(r-3)$-dimensional subspace $z$ in $\PG(r-1,q^n)$,
such that $l_\infty\cap z=\emptyset$, the map
$x\mapsto \langle x,z\rangle$ defines a projectivity
from $\PG(r-1,q^n)\setminus z$ to $\PG(r-1,q^n)/z$; 
hence it maps linear sets into linear sets and non-linear sets into non-linear sets;\\
(iii) if $\pi_0$ is a $q$-subgeometry of $\pi\cong\PG(r-1,q^n)$, $r>2$, then the set $\cH(\pi_0)$ 
of hyperplanes of $\pi_0$ can be identified with
the set $\cP(\pi_0^d)$ of points of a $q$-subgeometry $\pi_0^d$ of $\pi^d$.

If $r=2$ the statement is trivial, since $S$ is a $q$-subline, and it is splash of itself.
In the following $r>2$ is assumed.

{\bf A.} First we prove that each splash is a linear set.

Let $S$ denote a splash on the
line $l_\infty$, defined by the $q$-subgeometry $\pi_0$ of $\pi\cong{\mathrm{PG}}(r-1,q^n)$, i.e.
\begin{eqnarray}
S=\{ l_\infty \cap \bar h ~:~ h \in {\mathcal{H}}(\pi_0)\}.
\end{eqnarray}
The dual of $S$ is
\begin{eqnarray}
S^d=\{ \langle l_\infty ^d, \bar h^d\rangle ~:~h \in {\mathcal{H}}(\pi_0)\}=\{ \langle l_\infty ^d, x\rangle ~:~x \in {\mathcal{P}}(\pi_0^d)\}.
\end{eqnarray}
Note that if a point of $\pi_0^d$ were on $l_\infty^d$, then the line $l_\infty$ would
be contained in the extension of a hyperplane of $\pi_0$, contradicting the definition of a splash.
Now consider the projection $\Psi$ of $\pi^d\setminus l_\infty^d$ onto the 
quotient space $\pi^d/l_\infty^d$.
Then $S^d=\Psi(\cP(\pi_0^d))$ is
projectively equivalent (by (ii) above) to
the projection with vertex $l_\infty^d$ of the subgeometry $\pi_0^d$ onto a line 
disjoint from $l_\infty^d$, and hence is a linear set by \cite[Theorem 2]{LuPo04}.

\noindent
{\bf B.} Next we prove that each linear set of rank $r$ on $\PG(1,q^n)$ is a splash.

Let $S$ be an ${\mathbb{F}}_q$-linear set of rank $r$ on the line $l_\infty\cong\PG(1,q^n)$.
Embed $l_\infty$ as a line in $\pi={\mathrm{PG}}(r-1,q^n)$. 

Consider an arbitrary $(r-3)$-dimensional subspace $z$ disjoint from $l_\infty$ in $\pi$.
The set of hyperplanes
$L=\{\langle z,x\rangle ~:~ x \in S\}$, which is projectively equivalent to $S$,
defines an ${\mathbb{F}}_q$-linear set $L^d$ of rank $r$ in the dual space $\pi^d$, 
contained in the line $z^d\cong {\mathrm{PG}}(1,q^n)$. 
Hence there exists a subgeometry $\pi_0\cong\PG(r-1,q)$, such
that $L^d$ is the projection of $\pi_0^d$ from an $(r-3)$-dimensional subspace $l^d$
sharing no dual point with $\pi_0^d$ (implying that no extension of a hyperplane of $\pi_0$ contains $l$), i.e.
$L^d=\{ z^d\cap \langle x, l^d\rangle~:~ x \in {\mathcal{P}}(\pi_0^d)\}$.

Equivalently we have
$L=\{\langle z, \bar H\cap l\rangle~:~ H \in {\mathcal{H}}(\pi_0)\}$.
This implies that $S$ is the projection from $z$ onto the line $l_\infty$ of the splash $\{\bar H\cap l~:~H \in {\mathcal{H}}(\pi_0)\}$ of
$\pi_0$ on the line $l$, and hence $S$ is a splash.
\end{proof}

In order to avoid the case of a $q$-subline, from now on $r$ will be an integer greater than two, unless otherwise stated (cf.\ prop.\ \ref{prop:unique}).

\begin{theorem}\label{thm:weight}
Let $S$ be the splash of a subgeometry $\pi_0\cong\PG(r-1,q)$ of $\pi\cong\PG(r-1,q^n)$ on $l_\infty\cong\PG(1,q^n)$. The following statements are equivalent.
\begin{itemize}
\item[(i)] The point $x\in S$ has weight $j$.
\item[(ii)] There are $(q^j-1)/(q-1)$ hyperplanes of $\pi_0$ through $x\in S$.
\end{itemize}
\end{theorem}
\begin{proof} Put $\theta_j:=(q^j-1)/(q-1)$.
Theorem \ref{thm:weight} is the dualization of the fact that in the representation of a linear set 
as a projection of a $q$-subgeometry $\pi_0$ \cite[Theorem 2]{LuPo04}, a point $x$ has weight $j$ if, 
and only if, $x$ is projection of precisely $\theta_j$ points of $\pi_0$.
Suppose $x\in S$ has weight $j$. Arguing as in the proof of Theorem \ref{thm:equivalence}, 
this implies that there are $\theta_j$ points of $\pi_0^d$ which project onto $x^d$ from $l_\infty^d$, and hence
$x$ is contained in $\theta_j$ hyperplanes of $\pi_0$.
Conversely, consider the set $\cH(x)$ of hyperplanes of $\pi_0$ on a point $x\in S$, and suppose 
$\cH(x)=\{h_1,\ldots, h_{\theta_j}\}$. This means that
$x^d=\langle h_i^d,l_\infty^d\rangle$ for $i\in\{1,\ldots ,{\theta_j}\}$, and hence that $x$ has weight $j$.
\end{proof}

\begin{cor}
Let $S$ be the splash of a subgeometry $\pi_0\cong\PG(r-1,q)$ of $\pi\cong\PG(r-1,q^n)$ on $l_\infty\cong\PG(1,q^n)$. Then $S$ is a scattered linear set if and only if $S$ is an external splash, where every point of $S$ is on exactly one hyperplane of $\pi_0$.
\end{cor}
\begin{proof}
If $S$ is scattered then each point has weight 1. The rest of the proof is immediate from Theorem \ref{thm:equivalence} and Theorem \ref{thm:weight}.
\end{proof}

\section{Characterization of clubs}\label{sect:Characterization}

If $P_1$, $P_2$ and $P_3$ are distinct collinear points in some projective space $\PG(m,q^n)$,
then the unique $q$-subline containing them is denoted by $\subl_q(P_1,P_2,P_3)$.
The aim of this section is to prove the following characterization.

\begin{theorem}\label{thm:zero}
  Let $T$ be a point and $\mathcal A$ a $q^{r-1}$-set, $3\leq r\le n$,  in ${\mathrm{PG}}(1,q^n)$ such that
  $T\not\in{\mathcal A}$. 
  Consider the following three statements.
  \begin{enumerate}
    \item[$(i)$] $T\cup{\mathcal A}$ is an ${\mathbb{F}}_q$-club, and $T$ has weight $r-1$;
    \item[$(ii)$] $T\cup{\mathcal A}$ is a tangent splash with centre $T$;
    \item[$(iii)$] for any pair of distinct points $P,Q\in{\cal A}$, the subline
    ${\mathrm{subl}}_q(T,P,Q)$ is contained in $T\cup{\cal A}$.
  \end{enumerate}
  Then the statements $(i)$ and $(ii)$ are equivalent, and if $q>2$ all three statements are equivalent.
\end{theorem}
\begin{proof}

\underline{$(i)\Leftrightarrow (ii)$} The equivalence of the first two statements easily follows from Theorem \ref{thm:equivalence} and Theorem \ref{thm:weight}. Namely, if 
$T\cup{\mathcal A}$ is an ${\mathbb{F}}_q$-club, and $T$ has weight $r-1$, then $T\cup{\mathcal A}$ is a splash of
a subgeometry $\pi_0\cong\PG(r-1,q)$ of $\pi\cong\PG(r-1,q^n)$ on $l_\infty\cong\PG(1,q^n)$. 
By Theorem \ref{thm:weight} there are $(q^{r-1}-1)/(q-1)$ extended hyperplanes of $\pi_0$ through $T$. This implies $T\in {\mathcal{P}}(\pi_0)$. Since all other points have weight one, they lie on exactly one hyperplane of $\pi_0$ and we may conclude that $\pi_0$ is tangent to $l_\infty$. This proves the implication $(i)\Rightarrow (ii)$. Similarly $(ii)\Rightarrow(i)$.

\underline{$(i)\Rightarrow (iii)$} 
Consider the field reduction map $\cF:=\cF_{2,n,q}$.
Suppose $T\cup{\mathcal A}$ is an ${\mathbb{F}}_q$-club, and $T$ has weight $r-1$. This implies that there exists an $(r-1)$-dimensional subspace $U$ in $\PG(2n-1,q)$ such that $\cB(U)=T\cup{\mathcal A}$ and ${\mathrm{dim}}(\cF(T)\cap U)=r-2$. Now consider a subline ${\mathrm{subl}}_q(T,P,Q)$ for two distinct points $P, Q \in \cA$. This subline corresponds to the regulus determined by $\cF(T)$, $\cF(P)$ and $\cF(Q)$ in $\PG(2n-1,q)$. Since ${\mathrm{dim}}(\cF(T)\cap U)=r-2$, the line $m:=\langle \cF(P)\cap U,\cF(Q)\cap U\rangle$ meets $T$ and hence is a transversal to the regulus corresponding to ${\mathrm{subl}}_q(T,P,Q)$. As $m\subset U$ it follows that ${\mathrm{subl}}_q(T,P,Q)=\cB(m) \in \cB(U)=T\cup \cA$.

\underline{$(iii)\Rightarrow (i)$}
Now assume $(iii)$ holds and $q>2$. Choose a point $X \in \cA$ and a point $x$ in $\cF(X)$. Since for each $Y \in \cA\setminus \{X\}$, the $q$-subline $\subl_q(T,X,Y)$ is
contained in $T\cup{\cA}$, there exists a unique line $l_Y$ through $x$ in $\PG(2n-1,q)$ such that $\subl_q(T,X,Y)=\cB(l_Y)$. Let $U$ denote the union of the points on the $(q^{r-1}-1)/(q-1)$ lines defined in this way, i.e.
$$U:=\underset{Y\in \cA\setminus\{X\}}{\bigcup}{\mathcal{P}}(l_Y).$$
Then $|U|=(q^r-1)/(q-1)$ and $(q^{r-1}-1)/(q-1)$ points of $U$ are contained in $\cF(T)$. Moreover $T\cup\cA = \cB(U)$. 
Put $W:=\langle x, \cF(T)\rangle$. Then $W$ has dimension $n$ and $U\subset W$.

Each $\cF(C)$, with $C\in \cA$ intersects $W$ in exactly one point, and since $U$ contains a point of $\cF(C)$, and $U\subset W$,
it must hold $\cF(C)\cap W\in U$. Hence 
\begin{eqnarray}\label{eqn:formula}
\{\cF(C)\cap W~:~C\in \cA\}=U\setminus \cF(T).
\end{eqnarray}


Since $r>2$, there is a line $m$, $x\notin m$, spanned by two points $x_1,x_2$ of $U\setminus \cF(T)$. 
Then by hypothesis $\subl_q(\cB(x_1),\cB(x_2),T)$ is contained in $T\cup \cA$. 
Hence the points of $m$ not in $\cF(T)$ are contained in $U$, since 
by (\ref{eqn:formula}), $m\setminus \cF(T)=\cF(\subl_q(\cB(x_1),\cB(x_2),T)\setminus \{T\})\cap W$.

But then the lines $l_{\cB(y)}=\langle x, y \rangle$ with $y\in m\setminus (m\cap \cF(T))$ 
must be contained in $U$, implying that $U$ contains the affine plane $\langle x , m \rangle \setminus \langle x, m\cap \cF(T)\rangle$.
Since $q>2$, we may repeat the arguments for another line $m'$, lying
in the plane  $\langle x,m\rangle$ and going through $x_1$, and conclude that $U$ contains the plane 
$\langle x , m \rangle$, and this plane meets $\cF(T)$ in a line. In particular we have shown that $U$ contains every line spanned by two points
$x_1,x_2$ of $U\setminus \cF(T)$.
Since $|U|=(q^r-1)/(q-1)$ and $(q^{r-1}-1)/(q-1)$ points of $U$ are contained in $\cF(T)$, it follows that $U$ is an
$(r-1)$-dimensional subspace in $\PG(2n-1,q)$ such that $\cB(U)=T\cup{\mathcal A}$ and ${\mathrm{dim}}(\cF(T)\cap U)=r-2$.
Equivalently, $\cA\cup T$ is an $\Fq$-linear set of rank $r$, and $T$ is a point of weight $r-1$. This proves the implication $(iii)\Rightarrow (i)$.
\end{proof}

\begin{remark}
If $q=2$, statement $(iii)$ is always satisfied, and not every set of $2^{r-1}+1$ points is a linear set. 
\end{remark}

The linearity gives the advantage of having further almost straightforward consequences
concerning uniqueness and number of tangent splashes.
This allows to generalize the results in \cite[Sect.\ 5]{BaJa14}.
A tangent splash of a $q$-subgeometry of $\PG(r-1,q^n)$ is said to have \textit{rank} $r$
because it is indeed an $\F_q$-linear set of rank $r$. 
For the purpose of the 
following proposition only, a $q$-subline is called a {\it tangent splash of rank 2}
and any point on it is a centre.

\begin{prop}\label{prop:unique}
If $T$, $U_1, \ldots, U_r$ are distinct points in $l_\infty=\PG(1,q^n)$, $3\leq r \leq n$, and 
no $U_j$, $j\geq 3$, is contained in an $\F_q$-tangent splash of rank less than $j$ with centre $T$ containing the points
$U_1,\ldots, U_{j-1}$, then there is a unique tangent splash $S(\pi_0,l_\infty)$ of 
a $q$-subgeometry $\pi_0$ of $\PG(r-1,q^n)$, such that  $S(\pi_0,l_\infty)$ 
contains $U_1,\ldots, U_r$ and has centre $T$. 
\end{prop}
\begin{proof}
Consider the field reduction map $\cF:=\cF_{2,n,q}$. Let $u_1$ be a point of $\cF(U_1)$, and suppose
$T\cup \cA$ is a tangent splash of rank $r$ with centre $T$ and containing the points $U_1, \ldots, U_r$. Then
$T\cup \cA$ is a linear set, say $\cB(W)$, with $W$ an $(r-1)$-dimensional subspace of $\PG(2n-1,q)$, which
intersects $\cF(T)$ in a subspace of dimension $r-2$. By \cite[Lemma 4.3]{LaVaPrep}
we may assume $u_1 \in W$.
Put $u_i:=W\cap U_i$, for $i=2,3,\ldots, r$.
Since each line $\langle u_1,u_j\rangle$ is contained in $W$ and meets $\cF(T)$, it follows that 
$W$ must contain the unique transversal through $u_1$ to the regulus determined by 
$\cF(T),\cF(U_1),\cF(U_j)$, $j\neq 1$.
By the assumption that no $U_j$, $j\geq 3$, is contained in a tangent splash of rank $j-1$ with centre $T$ containing the points $U_1,\ldots, U_{j-1}$, it follows that the subspace $\langle u_1, u_2, \ldots, u_r \rangle$ has dimension
$r-1$, and hence must equal $W$. By Theorem \ref{thm:zero} this implies both existence and uniqueness.
\end{proof}

In the case of an $\F_q$-subplane tangent splash, as a corollary of the  Proposition \ref{prop:unique}, the following generalization of \cite[Theorem 5.1]{BaJa14} is obtained.

\begin{theorem}\label{thm:enumeration1}
  If $T$, $U$, $V$, and $W$ are distinct points in $\PG(1,q^n)$, and
  $W\not\in\subl_q(T,U,V)$, then a unique tangent splash of a $q$-subplane exists which contains
  $U$, $V$, $W$ and has centre $T$.
\end{theorem}

The following proposition gives the number of tangent splashes on $\PG(1,q^n)$ obtained from
order $q$-subgeometries in $\PG(r-1,q^n)$. This proposition generalizes \cite[Theorem 5.2]{BaJa14}.
\begin{prop}
  Let $r\ge3$.  
  \begin{enumerate}
  \item[$(i)$] The number of distinct rank $r$ tangent splashes of $q$-subgeometries on $\PG(1,q^n)$ 
  having a common centre $T$ is
  \begin{equation}\label{e:distinct-ts-common-centre}
   q^{n+1-r}\prod_{i=0}^{r-2}\frac{\displaystyle q^{n-i}-1}{\displaystyle q^{r-1-i}-1}.
  \end{equation}
  \item[$(ii)$] The number of distinct rank $r$ tangent splashes  of $q$-subgeometries on $\PG(1,q^n)$ is
  \begin{equation}\label{e:distinct-ts}
    (q^n+1)q^{n+1-r}\prod_{i=0}^{r-2}\frac{\displaystyle q^{n-i}-1}{\displaystyle q^{r-1-i}-1}.
  \end{equation}
  \end{enumerate}
\end{prop}

\begin{proof}
A tangent splash of rank $i$ has $1+q^{i-1}$ points.
(\textit i) 
The number of $r$-tuples $(U_1,U_2,\ldots,U_r)$ satisfying the assumptions of prop.\ \ref{prop:unique}
is
\begin{equation}\label{e:dc1}
  K=q^{n}\cdot(q^{n}-1)\cdot(q^{n}-q)\cdot\ldots\cdot(q^{n}-q^{r-2}).
\end{equation}
If $N$ is the number of tangent splashes with center $T$, then
\begin{equation}\label{e:dc2}
  K=Nq^{r-1}\cdot(q^{r-1}-1)\cdot(q^{r-1}-q)\cdot\ldots\cdot(q^{r-1}-q^{r-2}).
\end{equation}
Equations (\ref{e:dc1}) and (\ref{e:dc2}) imply (\ref{e:distinct-ts-common-centre}).
The total number of tangent splashes of rank $r$ is $(q^n+1)N$ and this proves (\ref{e:distinct-ts}).
\end{proof}

\section{Projective equivalence of tangent splashes}\label{sec:projective_equivalence}

\begin{prop}\label{prop:alg_description_TS}
  Let $T\cup{\cal A}$ be a tangent splash of the $q$-subgeometry $\pi_0$ 
  of $\PG(r-1,q^n)$ on the line $l_\infty=\PG(U)$, with centre $T$.
  Let $P$ be any point of $\cal A$.
  Then there exist $u,v\in U$ and $\rho_1,\rho_2,\ldots,\rho_{r-2}\in\mathbb F_{q^n}$, such that
  $1$, $\rho_1$, $\rho_2$, $\ldots$, $\rho_{r-2}$ are linearly independent over $\F_q$,
  $\langle v\rangle_{q^n}=P$, and
  \begin{eqnarray}\label{eqn:alg_description_TS}
    {\cal A}=\left\{\left\langle xu+\sum_{i=1}^{r-2}y_i\rho_i u+v\right\rangle_{q^n}\mid 
    x,y_i\in\mathbb F_q,\,i=1,2,\ldots,r-2\right\}.
  \end{eqnarray}
\end{prop}
\begin{proof}
  By theorem \ref{thm:zero},  there is an
  $r$-dimensional $\mathbb F_q$-subspace of $U$, say $V$, such that
  $T\cup{\cal A}={\cal B}(V)$.
  From $\wt(T)=r-1$ it follows that some $u\in U$ and  
  $\rho_1,\rho_2,\ldots,\rho_{r-2}\in\mathbb F_{q^n}$, exist such that
  $1$, $\rho_1$, $\rho_2$, $\ldots$, $\rho_{r-2}$ are linearly independent over $\F_q$;
   $u,\rho_i u\in V$, $i=1,2,\ldots,r-2$, and $T=\langle u\rangle_{q^n}$.
  Taking the vector $v\in V$ such that $\langle v\rangle_{q^n}=P$ yields
  \[
    T\cup{\cal A}=\left\{\left\langle x'u+\sum_{i=1}^{r-2}y_i'\rho_i u+z'v\right\rangle_{q^n}\mid 
    x',y'_i,z'\in\mathbb F_q,\,i=1,2,\ldots,r-2\right\},
  \]
  and this implies (\ref{eqn:alg_description_TS}).
\end{proof}

\begin{prop}\label{prop:u_i}
  Let $\pi_0$ be a $q$-subgeometry of
  $\PG(W)\cong\PG(r-1,q^n)$ and let $l_\infty=\PG(U)$ be a line
  of $\PG(W)$ tangent to $\pi_0$.
  Assume that $S(\pi_0,l_\infty)=T\cup{\cal A}$, and that the notation in prop.\ \emph{\ref{prop:alg_description_TS}} holds.  
  Let $H_0$ be the hyperplane of $\pi_0$ such that $P\in\overline{H_0}$.
  Then an ordered $(r-1)$-tuple\footnote{Here $(r-1)$-tuples are considered as column vectors.} 
  $s=(s_0\ s_1\ \ldots\ s_{r-2})^T\in W^{r-1}$ exists such that
  (i)~$H_0=\cB(\langle s_0,s_1,\ldots,s_{r-2}\rangle_{q})$, 
  (ii)~$v=s_0+\sum_{i=1}^{r-2}\rho_i s_i$, and
  (iii)~$\pi_0={\cal B}\left(\langle u,s_0,\ldots,s_{r-2}\rangle_{q}\right)$.
  If $\gcd(n,r-1)=1$, then there is a unique $s\in W^{r-1}$ satisfying (i) and (ii).
\end{prop}
\begin{proof}
  \emph{Existence.}
  Since $T=\langle u\rangle_{q^n}$, an $r$-dimensional $\F_q$-subspace $V_0$ of $W$
  exists such that $\cB(V_0)=\pi_0$ and $u\in V_0$.
  With $r-1$ independent points of $H_0$, vectors
  $z_0,z_1,\ldots,z_{r-2}\in V_0$
  are associated,
  hence $\langle u,z_0,z_1,\ldots,z_{r-2}\rangle_{q}=V_0$, and
  $v,z_0,z_1\ldots,z_{r-2}$ are linearly dependent on $\F_{q^n}$.
  As a consequence $\xi_0,\xi_1,\ldots,\xi_{r-2}\in\F_{q^n}$ exist such that
  $v=\sum_{j=0}^{r-2}\xi_jz_j$.
  For any $j$ let $l_j$ be the line joining $T$ and $\langle z_j\rangle_{q^n}$.
  Since any point in $\cal A$ lies on a 
  hyperplane of $\PG(r-1,q^n)$ joining $r-1$ points $P_j$, with $P_j\in \pi_0\cap l_j\setminus\{T\}$,
  $j=0,1,\ldots,r-2$, for any $x,y_1,\ldots,y_{r-2}\in\F_q$ there exist $\alpha_j\in\F_q$,  
  $j=0,1,\ldots,r-2$, such that the vectors
  $xu+\sum_{i=1}^{r-2}y_i\rho_i u+v$, $\alpha_0u+z_0$, $\ldots$, $\alpha_{r-2}u+z_{r-2}$
  are linearly dependent over $\F_{q^n}$.
  Since $u$, $z_0$, $z_1$, $\ldots$, $z_{r-2}$ are linearly independent over $\F_{q^n}$,
  \[
    \det\begin{pmatrix}x+\sum_{i=1}^{r-2}y_i\rho_i &\xi_0&\xi_1&\ &\ldots&\ &\xi_{r-2}\\ 
    \alpha_0&1&0&&\ldots&&0\\ \alpha_1&0&1&&\ldots&&0\\ 
    \vdots&&&&&\\ \alpha_{r-2}&0&0&&\ldots&&1
    \end{pmatrix}=0,
  \]
  whence $\langle 1,\rho_1,\ldots,\rho_{r-2}\rangle_{q}=\langle\xi_0,\xi_1,\ldots,\xi_{r-2}\rangle_{q}$, and
  an $A\in\GL(r-1,q)$ exists such that
  \[(\xi_0\ \xi_1\ \ldots\ \xi_{r-2})=(1\ \rho_1\ \ldots\ \rho_{r-2})A.\]
  By defining $s=(s_0\ s_1\ \ldots\ s_{r-2})^T=A(z_0\ z_1\ \ldots\ z_{r-2})^T$,
  $(i)$ and $(iii)$ are straightforward.
  Furthermore,
  \[
    v=(\xi_0\ \xi_1\ \ldots\ \xi_{r-2})(z_0\ z_1\ \ldots\ z_{r-2})^T=(1\ \rho_1\ \ldots\ \rho_{r-2})s
  \]
  and this is $(ii)$.
  
  \emph{Uniqueness.}
  Let $\rho=(1\ \rho_1\ \ldots\ \rho_{r-2})^T$.
  Assume that \[s=(s_0\ s_1\ \ldots\ s_{r-2})^T,\,s'=(s'_0\ s'_1\ \ldots\ s'_{r-2})^T\in W^{r-1}\] satisfy
  $H_0={\cal B}\left(\langle s_0,s_1,\ldots,s_{r-2}\rangle_{q}\right)
  ={\cal B}\left(\langle s'_0,s'_1,\ldots,s'_{r-2}\rangle_{q}\right)$, and
  $v=\rho^Ts=\rho^Ts'$.
  If $V$ and $V'$ are subspaces of $W$ such that $\cB(V)=\cB(V')$ is a $q$-subgeometry $K_0$ of $\PG(W)$, 
  then the related projective subspaces in $\PG(rn-1,q)$ are subspaces of the same family of maximal
  subspaces of the Segre variety $\cF(K_0)$
  (see \cite[Theorem 2.4]{LaVaPrep}), and a $\zeta\in\F_{q^n}^*$ exists such that
  $\zeta V'=V$. Therefore, the assumptions imply that a $\zeta\in\F^*_{q^n}$
  and an $M\in\GL(r-1,q)$ exist such that
  \begin{equation}\label{eqn:zeta}\zeta s'=Ms.\end{equation}
  From $\zeta\rho^Ts'=\zeta\rho^Ts$ and (\ref{eqn:zeta}) one obtains
  $\zeta\rho^Ts=\rho^TMs$.
  Since $s_0,s_1,\ldots,s_{r-2}$ are linearly independent vectors on $\F_{q^n}$,
  the last equation implies $M^T\rho=\zeta\rho$.
  As a consequence, for any $j\in\mathbb N$, $\rho^{q^j}$ is an eigenvector of $M^T$
  related to the eigenvalue $\zeta^{q^j}$.
  The dimension $d_j$ of the related eigenspace $\{w\in\F_{q^n}^{r-1}\,:\, M^Tw=\zeta^{q^j}w\}$
  does not depend on $j$, so let $d=d_j$.
  It holds in general that if $x_0,x_1,\ldots,x_{r-2}\in\F_{q^n}$ are linearly independent
  over $\F_q$, then the $r-1$ vectors $(x_0^{q^j}\ x_1^{q^j}\ \ldots\ x_{r-2}^{q^j})^T$,
  $j=0,1,\ldots,r-2$,
  are linearly independent over $\F_{q^n}$, and vice-versa.
  (\cite[Lemma 3.51, p. 109]{LiNi}).
  
  Hence $\rho$, $\rho^q$, $\ldots$, $\rho^{q^{r-2}}$ are linearly independent eigenvectors of $M^T$, and
  $M^T$ is similar to the matrix $M'=$diag$(\zeta,\zeta^q,\ldots,\zeta^{q^{r-2}})$.
  Let $e=[\F_q(\zeta):\F_q]$, then the diagonal of $M'$ contains $e$ distinct elements,
  each repeated $d$ times, hence $r-1=de$; but $e$ divides $n$, and the assumption
  $\gcd(n,r-1)=1$ implies $e=1$ and $\zeta\in\F_q$.
  As a consequence of this result,  $M'=\zeta I_{r-1}=M$, and
  finally (\ref{eqn:zeta}) gives $s'=s$.
\end{proof}

A first consequence of Proposition\ \ref{prop:u_i} is a uniqueness theorem for $q$-subgeometries
giving rise to a tangent splash, partly generalizing \cite[Theorem 5.1]{BaJa14}:
\begin{theorem}\label{thm:uniqueness}
  Let $\pi_0$, $\pi_1$ be $q$-subgeometries of
  $\PG(W)\cong\PG(r-1,q^n)$, and let $l_\infty=\PG(U)$ be a line
  of $\PG(W)$ tangent to both $\pi_0$ and $\pi_1$.
  Assume $\gcd(n,r-1)=1$.
  If $S(\pi_0,l_\infty)=S(\pi_1,l_\infty)$ and a $q$-subgeometry $H_0$ of a hyperplane
  in $\PG(W)$ exists such that $H_0\subset\pi_0\cap\pi_1$,
  and $T\not\in\overline{H_0}$, then $\pi_0=\pi_1$.
\end{theorem}
\begin{proof}
  Assume that the point $P$ in prop.\ \ref{prop:alg_description_TS} belongs to $H_0$.
  By prop.\ \ref{prop:u_i} conditions $(i)$ and $(ii)$ imply $(iii)$, whence 
  $\pi_0={\cal B}\left(\langle u,s_0,\ldots,s_{r-2}\rangle_{q}\right)=\pi_1$.
\end{proof}
The generalization above cannot be extended to any $n$;
on the contrary, the uniqueness part of the proof of prop.\ \ref{prop:u_i} suggests how to
construct distinct $q$-subgeometries having a common hyperplane and same tangent splash.
In particular, for even $n$ it is not true that two $q$-subplanes of $\PG(2,q^n)$
having the same tangent splash
and sharing a $q$-subline not through $T$ must coincide.
\begin{theorem}\label{thm:non_uniqueness}
  Assume $d=\gcd(n,r-1)>1$.
  Then two distinct $q$-subgeometries of $\PG(W)\cong\PG(r-1,q^n)$,
  say $\pi_0$ and $\pi_1$, a line $l_\infty$ of $\PG(W)$ tangent to both $\pi_0$ and $\pi_1$
  at a point $T$, and a common hyperplane $H_0$ to $\pi_0$ and $\pi_1$ with $T\not\in\overline H$ exist such that
  $S(\pi_0,l_\infty)=S(\pi_1,l_\infty)$.
\end{theorem}
\begin{proof}
  Let $\zeta\in\F_{q^n}$ such that $[\F_q(\zeta):\F_q]=d$, and let $M_0\in\GL(d,q)$ be the companion
  matrix having as characteristic polynomial the minimal polynomial of $\zeta$.
  Let $w=(w_1\ w_2\ \ldots\ w_d)^T\in\F_{q^n}^d$ be an eigenvector of $M_0$ related to $\zeta$.
  Then $w$, $w^q$, $\ldots$, $w^{q^{d-1}}$ are linearly independent over $\F_{q^n}$,
  and this implies that $w_1$, $w_2$, $\ldots$, $w_d$ are linearly independent over $\F_q$.
  Therefore $w_1=1$ may be assumed.
  Since $w^{q^d}$ is an eigenvector of the one-dimensional eigenspace related to $\zeta$,
  it follows $w^{q^d}=w$, and $\langle1=w_1,w_2,\ldots,w_d\rangle_q=\F_{q^d}=\F_q(\zeta)$.
  Next, let $1=\omega_1$, $\omega_2$, $\ldots$, $\omega_{n/d}$ be an $\F_{q^d}$-basis
  of $\F_{q^n}$.
  Then $\{\omega_i w_j\mid i=1,2,\ldots,n/d,\ j=1,2,\ldots,d\}$ is an $\F_q$-basis of $\F_{q^n}$, and
  \[
    \rho=(1\ w_2\ \ldots\ w_d\ \omega_2\ \omega_2w_2\ \ldots\ \omega_2w_d\ \ldots\ \omega_{(r-1)/d}w_d)^T
    \in \F_{q^n}^{r-1}
  \]
  is an eigenvector of
  \[
  M=
  \begin{tikzpicture}[baseline,decoration={brace,amplitude=5pt}]
    \matrix (magic) [matrix of math nodes,left delimiter=(,right delimiter=).] {    
    M_0&&&\\ &M_0&&\\ &&\ddots&\\ &&&M_0\\
    };
    \draw[decorate,black] (magic-1-1.north) -- (magic-4-4.north) node[above=5pt,midway,sloped] {$(r-1)/d$};
  \end{tikzpicture}
  \]
  Now let $T\cup\cA$ be an $\F_q$-club of rank $r$ and head $T$ in $l_\infty$,
  satisfying (\ref{eqn:alg_description_TS}) with $\rho_1=w_2$, $\rho_2=w_3$, $\ldots$, 
  $\rho_{r-2}=\omega_{(r-1)/d}w_d$.
  By theorem \ref{thm:zero}, $T\cup\cA$ is a tangent splash $S(\pi_0,l_\infty)$ for some $q$-subgeometry
  $\pi_0$ of a projective space $\pi\cong\PG(r-1,q^n)$.
  Take $s\in W^{r-1}$ as in prop.\ \ref{prop:u_i}, and define $\zeta s'=M^Ts$, implying
  $\zeta\langle s'_0,s'_1,\ldots,s'_{r-2}\rangle_q=\langle s_0,s_1,\ldots,s_{r-2}\rangle_q$.
  Furthermore, $\pi_1:=\cB(\langle u,s'_0,\ldots,s'_{r-2}\rangle_q)\neq\pi_0$ since $\zeta\not\in\F_q$.
  Define $H_0=\cB\left(\langle s_0,\ldots, s_{r-2}\rangle_{{q}}\right)$; 
  clearly $H_0\subset\pi_0\cap\pi_1$ and $T\not\in\overline{H_0}$.  
  Let $\kappa\in\GL(r,q^n)$ be defined by $\kappa(u)=u$, $\kappa(s_i)=s'_i$, $i=0,1,\ldots,r-2$.
  The related projectivity $\hat{\kappa}$ maps $\pi_0$ onto $\pi_1$.
  Furthermore
  \[
    \kappa(v)=\rho^Ts'=\zeta^{-1}\rho^TM^Ts=\rho^Ts=v,
  \]
  so the restriction of $\hat\kappa$ to $l_\infty$ is the identity, whence 
  $S(\pi_0,l_\infty)=S(\pi_1,l_\infty)$.
\end{proof}

\begin{prop}\label{prop:same_TS}
  Let $\pi_0$ and $\pi_1$ 
  be two $q$-subgeometries of  
  $\PG(W)\cong\PG(r-1,q^n)$,
  both tangent to a line $l_\infty$ in $\PG(W)$.
  If $S(\pi_0,l_\infty)=S(\pi_1,l_\infty)$, then a $\hat{\kappa}\in\PGL(r,q^n)$ exists such that 
  $l_\infty^{\hat{\kappa}}=l_\infty$, $\pi_0^{\hat{\kappa}}=\pi_1$.
\end{prop}
\begin{proof}
  Assume $S(\pi_0,l_\infty)=S(\pi_1,l_\infty)=T\cup{\cal A}$ as described in prop.\ \ref{prop:alg_description_TS}.
  By prop.\ \ref{prop:u_i}, $s,s'\in W^{r-1}$ exist such that
  $v=\rho^T s=\rho^T s'$, where $\rho=(1\ \rho_1\ \ldots\ \rho_{r-2})$,
   and
  $\pi_0={\cal B}\left(\langle u,s_0,\ldots,s_{r-2}\rangle_{q}\right)$,
  $\pi_1={\cal B}\left(\langle u,s'_0,\ldots,s'_{r-2}\rangle_{q}\right)$.
  Take $\kappa\in\GL(r,q^n)$ defined by $u^\kappa=u$, $s_i^\kappa=s'_i$, $i=0,1,\ldots,r-2$.
  By construction the associated projectivity $\hat{\kappa}$ satisfies $\pi_0^{\hat{\kappa}}=\pi_1$, 
  $P^{\hat{\kappa}}=P$, $T^{\hat\kappa}=T$, hence $l_\infty^{\hat{\kappa}}=l_\infty$.
\end{proof}

By the next result, there is a bijection between orbits, with respect to $\PGL(2,q^n)$,
 of $\F_q$-clubs in $\PG(1,q^n)$
 and orbits,  with respect to $\AGL(r-1,q^n)$,
 of $q$-subgeometries in $\AG(r-1,q^n)$ tangent to the line
at infinity.

\begin{theorem}\label{thm:equivalenza}
  Let $\pi_0$ and $\pi_1$ 
  be two $q$-subgeometries of  
  $\PG(W)\cong\PG(r-1,q^n)$,
  both tangent to a line $l_\infty$ in $\PG(W)$.
  If a collineation $\theta\in\PGGL (2,q^n)$ exists such that $S(\pi_0,l_\infty)^\theta=S(\pi_1,l_\infty)$, 
  then
  a $\tau\in\PGGL(r,q^n)$ exists such that $l_\infty^\tau=l_\infty$ and $\pi_0^\tau=\pi_1$.
  Conversely, if $\tau\in\PGGL(r,q^n)$ exists such that $l_\infty^\tau=l_\infty$ and $\pi_0^\tau=\pi_1$
  then $S(\pi_0,l_\infty)^\tau=S(\pi_1,l_\infty)$.
  The assertions above still hold by substituting every $\PGGL (2,q^n)$ and $\PGGL (r,q^n)$
  with $\PGL(2,q^n)$ and $\PGL(r,q^n)$, respectively.
\end{theorem}
\begin{proof}
  The collineation $\theta$ can be extended to $\overline\theta\in\PGGL(r,q^n)$.
  From $S(\pi_0^{\overline\theta},l_\infty)=S(\pi_1,l_\infty)$, by prop.\ \ref{prop:same_TS}, a
  $\hat\kappa\in\PGL(r,q^n)$ exists such that $l_\infty^{\hat{\kappa}}=l_\infty$,
  $\pi_0^{\overline\theta\hat\kappa}=\pi_1$.
  Just take $\tau=\overline\theta\hat\kappa$ in order to obtain the first assertion.
  The remainder of the theorem is straightforward.
\end{proof}

\begin{remark}
Theorem \ref{thm:equivalenza} translates the classification problem of clubs, which remains open, 
in terms of AGL$(2,q^n)$.
Since for $n>3$ there exist projectively nonequivalent clubs \emph{\cite{LaVV10}}, from
Theorem \ref{thm:equivalenza} it follows that, with respect to 
  $\AGL(2,q^n)$,
there exist at least two non-equivalent projective $q$-subplanes of $\AG(2,q^n)$ that
  are tangent to the line at infinity.  
\end{remark}

\end{document}